\documentclass[12pt]{amsart}

\usepackage{amsmath,amssymb,a4}
\usepackage[greek,english]{babel}

\theoremstyle{plain}
\newtheorem{theorem}{Theorem}[section]
\newtheorem{lemma}[theorem]{Lemma}
\newtheorem{corollary}[theorem]{Corollary}
\newtheorem{proposition}[theorem]{Proposition}
\theoremstyle{definition}
\newtheorem{remark}[theorem]{Remark}

\newtheorem{definition}[theorem]{Definition}

\numberwithin{equation}{section}
\def\N{\mathbb N}
\def\R{\mathbb R}

\begin{document}

 \baselineskip=17pt    

\title{Characterization of the weak-type boundedness of the Hilbert transform on weighted Lorentz spaces} 
\author {Elona Agora, Mar\'{\i}a J. Carro,  and Javier Soria} 

\address{E. Agora, Departament de
Matem\`atica Aplicada i~An\`alisi, Universitat de Barcelona, 
 08007 Barcelona, Spain.}
\email{elona.agora@gmail.com}

\address{M. J. Carro, Departament de
Matem\`atica Aplicada i~An\`alisi, Universitat de Barcelona, 
 08007 Barcelona, Spain.}
\email{carro@ub.edu}

\address{J. Soria, Departament de
Matem\`atica Aplicada i~An\`alisi, Universitat de Bar\-ce\-lona, 
 08007 Barcelona, Spain.}
\email{soria@ub.edu}

\subjclass[2010]{26D10, 42A50}
\keywords{Weighted Lorentz spaces, Hilbert transform, Muckenhoupt weights, $B_p$ weights}
\thanks{This work was  partially supported by the Spanish Government Grant MTM2010-14946.} 

\begin{abstract} We   characterize the weak-type boundedness of the Hilbert transform $H$ on weighted Lorentz spaces
$\Lambda^p_u(w)$,  with $p>0$,   in terms of some geometric conditions on the weights $u$ and $w$ and the weak-type boundedness  of the Hardy-Littlewood maximal operator on the same spaces. Our results recover simultaneously the theory of the boundedness of $H$  on weighted Lebesgue spaces $L^p(u)$ and Muckenhoupt weights $A_p$,  and the theory on classical Lorentz  spaces $\Lambda^p(w)$ and Ari\~no-Muckenhoupt weights $B_p$. 
\end{abstract}

\date{\today}

\maketitle

\pagestyle{headings}\pagenumbering{arabic}\thispagestyle{plain}

\markboth{Weak-type boundedness of the Hilbert transform}{Elona Agora, Mar\'{\i}a J. Carro,  and Javier Soria }

\section{Introduction and motivation}

In this paper, we    characterize   the weak-type boundedness of the Hilbert transform on weighted Lorentz spaces
\begin{equation} \label{Hilbertantecedents}
H : \Lambda^p_u(w) \longrightarrow\Lambda^{p, \infty}_u(w),  
\end{equation}
 if $0<p<\infty$,  and $H$ is  the Hilbert transform defined by 
$$
Hf(x)=\frac{1}{\pi} \lim_{\varepsilon\to 0^+} \int_{|x-y| > \varepsilon} \frac{f(y)}{x-y}\,dy, 
$$
whenever this limit exists almost everywhere. We recall (see \cite{l1:l1, l2:l2}) that, given $u$,  a positive and locally integrable function (called weight) in $\mathbb R$ and given a weight $w$ in $\mathbb R^+$, the Lorentz space $\Lambda^{p}_{u}(w)$ is defined as 
$$
\Lambda^{p}_{u}(w) =\left\{f\in\mathcal M(\mathbb R): \, ||f||_{\Lambda^{p}_{u}(w)}=\left( \int_0^{\infty}  (f^*_u(t))^p w(t)dt\right) ^ {1/p}< \infty \right\},
$$
where  $\mathcal M=\mathcal M(\mathbb R)$ is the set of Lebesgue measurable functions on $\mathbb R$,  $f^*_u$  is  the decreasing rearrangement of $f$ with respect to the weight $u$  \cite{bs:bs}  
$$
f^*_u(t)=\inf \{ y>0: u(\{x\in\mathbb R: |f(x)|>y\}) \le t\}, 
$$
with $u(E)=\int_E u(x) dx$, and the weak-type Lorentz space is
$$
\Lambda^{p, \infty}_{u}(w) =\left\{f\in\mathcal M : \, ||f||_{\Lambda^{p, \infty}_{u}(w)}=\sup_{t>0} f^*_u(t) W(t)^ {1/p}< \infty \right\}, 
$$
where $W(t)=\int_0^t w(s)ds$.  In order to avoid trivial cases, we will assume that $u(x)>0$, a.e. $x\in\mathbb R$.

 The motivation for studying (\ref{Hilbertantecedents})  comes naturally, as a unified theory,  from the fact that weighted Lorentz spaces include, as particular examples, the weighted Lebesgue spaces $L^p(u)$ and    the classical Lorentz spaces $\Lambda^p(w)$,  and in both cases the boundedness of the Hilbert transform is already known \cite{hmw:hmw, cf:cf, s:s}. They also include the case of  the Lorentz spaces $L^{p, q}(u)$,  where only some partial results were previously known  \cite{chk:chk}.

\noindent
(i)  If $w=1$, \eqref{Hilbertantecedents} is equivalent to the fact that 
$$
H:L^p(u) \to L^{p, \infty}(u)
$$
is bounded,  and this problem was   solved by Hunt, Muckenhoupt,  and Wheeden  \cite{hmw:hmw}. An alternative proof was provided in \cite{cf:cf} by Coifman and Fefferman and the solution is the $A_p$ class of  weights,  if $p>1$ \cite{m:m}:
\begin{equation*}\label{defAp}
\sup_{I}\left(\frac{1}{|I|}\int_I u(x)dx\right) \left(\frac{1}{|I|} \int_I u^{-1/(p-1)}(x) dx \right)^{p-1}<\infty, 
\end{equation*}
where the supremum is considered over all intervals $I$ of the real line.

 This condition also characterizes the strong-type boundedness 
$$
H:L^p(u) \to L^{p}(u), 
$$
and if $p=1$ 
$$
H:L^1(u) \to L^{1,\infty}(u)
$$
is bounded if and only if $u\in A_1$: 
$$
Mu(x)\le C u(x), \quad \text{ a.e.} \  x\in\mathbb R, 
$$
with  $M$ being  the Hardy-Littlewood maximal function:
$$
Mf(x)=\sup_{x\in I}{\frac{1}{|I|}}\int_{I} |f(y)|dy,
$$
where the supremum is taken over all intervals $I$ containing $x\in \R$.  

Recall  \cite{gr:gr}
 that a weight $u\in A_\infty$ if and only if  there exist $C_u>0$ and  $\delta \in (0,1)$ such that,  for every interval $I$ and
every measurable set $E\subset I$,
\begin{equation} \label{Ainfty}
 \frac{u(E)}{u(I)}\le C_u \left(\frac{|E|}{|I|}\right)^{\delta}, 
\end{equation}
and  it holds that 
$$
A_{\infty}= \bigcup _{p\ge 1} A_p.
$$

\noindent
(ii)  On the other hand, if $u=1$, the characterization of \eqref{Hilbertantecedents} is equivalent to the boundedness of
$$
H:\Lambda^p(w) \longrightarrow  \Lambda^{p, \infty}(w),
$$
given by Sawyer  \cite{s:s}. A simplified description of the class of weights \cite{n:n} that characterizes this property is $B_{p, \infty}\cap B^*_{\infty}$,  where a weight  $w\in  B^*_{\infty}$ if 
\begin{equation}\label{binfty}
\int_0^r \frac{1}{t}\int_{0}^t w(s)ds \ dt \le C \int_{0}^r w(s)ds,
\end{equation}
for all $r>0$, and $w\in B_{p, \infty}$ if the Hardy operator
$$
Pf(t)=\frac 1t\int_0^t f(s) ds
$$
satisfies that
$$
P: L^p_{\text{dec}}(w) \longrightarrow L^{p, \infty}(w)
$$
is bounded, where
$$
 L^p_{\text{dec}}(w)=\left\{ f\in  L^p (w): f \mbox{ is decreasing} \right\}. 
 $$
These weights have been well studied  (see \cite{am:am, n1:n1, crs:crs}) and it is known that   if $p\le 1$ then,  $w\in B_{p, \infty}$ if and only if 
$W$ is  $p$ quasi-concave:  for every $0<r<t<\infty$
$$
\frac{W(t)}{t^p} \le C \frac{W(r)}{r^p}, 
$$
and if $p>1$, $B_{p, \infty}=B_p$,  where $w\in B_p$   if 
\begin{equation}
r^p \int_r^{\infty}\frac{w(t)}{t^p} \, dt  \le C  \int_0^rw(s) ds
\end{equation}
for every $r>0$.  Moreover, for every $p>0$, 
$$
M:\Lambda^p(w) \longrightarrow  \Lambda^{p, \infty}(w),
$$
if and only if $w\in B_{p, \infty}$, 

If we consider the strong-type boundedness 
$$
H:\Lambda^p(w) \longrightarrow  \Lambda^p(w),
$$
this is equivalent to the condition  $w\in B_{p}\cap B^*_{\infty}$.

In   \cite{acs:acs}  we gave the following characterization of the weights $w$ for which (\ref{Hilbertantecedents})  holds under the assumption that $u\in A_1$: 
$$
H: \Lambda^p_u(w) \to \Lambda^{p, \infty}_u(w)  \Longleftrightarrow  w\in B_{p,\infty}\cap B^*_{\infty}, \qquad p>0. 
$$
We also proved that if $p>1$ and $u\in A_1$, then 
$$
H: \Lambda^p_u(w) \to \Lambda^{p}_u(w)  \Longleftrightarrow  w\in B_{p}\cap B^*_{\infty}.
$$

The main result of this paper solves the weak-type boundedness of $H$ for a general weight $u$,  as follows: 

\begin{theorem} \label{main} For every $0<p<\infty$, 
$$
H: \Lambda^p_u(w) \to \Lambda^{p, \infty}_u(w)
$$
is bounded if and only if the following conditions hold: 

\noindent
\mbox{\rm (i)} $u\in A_\infty$.

\noindent
\mbox{\rm (ii)} $w\in B^*_\infty$.

\noindent
\mbox{\rm (iii)}
$ 
M: \Lambda^p_u(w) \to \Lambda^{p, \infty}_u(w)
$ 
is bounded. 
\end{theorem}

\begin{remark}

The necessity of the   condition $u\in A_\infty$ in (i) was, for us,  an unexpected result since  in the case of the Hardy-Littlewood maximal operator it  was proved in \cite{crs:crs}  that   $u\in A_\infty$,  or even the doubling property, was not necessary to have the corresponding weak-type boundedness; that is
$$
M:\Lambda^p_u(w) \to \Lambda^{p, \infty}_u(w) \not\Rightarrow u\in A_\infty. 
$$

\end{remark}

\begin{remark}\label{rarara}
It is worth mentioning  that the   characterization of the weak-type  boundedness of the Hardy-Littlewood maximal operator in terms of the weights $u$ and $w$ was left open in \cite{crs:crs}, for $p\ge 1$. The   case $p<1$  is given by the following condition \cite{crs:crs}: for every finite family of disjoint intervals $\{I_j\}_{j=1}^J$,
and every family of measurable sets $\{S_j\}_{j=1}^{J}$, with $S_j\subset I_j$, for every $j$,  we have that
\begin{equation*}\label{raposo}
\frac{W\left(u\left(\bigcup_{j=1}^J I_j\right)\right)}{W\left(u\left(\bigcup_{j=1}^J S_j\right)\right)}
       \leq C\max_{1\leq j\leq J} \left(\frac{|I_j|}{|S_j|}\right)^p.
\end{equation*}
\end{remark}

\ 

We list now several results  that are  important for our purposes  \cite{acs:acs, crs:crs}:

\begin{proposition}\label{properties}

\noindent
{\rm (a)} $\Lambda^{p}_{u}(w)$ and $\Lambda^{p,\infty}_{u}(w)$  are quasi-normed spaces if and only if $w$ satisfies  the $\Delta_2$ condition; that is,  for every $r>0$, 
\begin{equation}\label{d2}
W(2r)\le C W(r).
\end{equation}

\noindent
{\rm (b)} If $u\notin L^1(\mathbb R)$,  $w\notin L^1(\mathbb R^+)$ and $w\in \Delta_2$, then ${\mathcal{C}}^{\infty}_c(\R)$  is  dense in $\Lambda^p_u(w)$.
\end{proposition}

 \begin{definition}
The  \textit{associate space}  of  $\Lambda^{p, \infty}_u(w)$, denoted as   $(\Lambda^{p, \infty}_u(w))'$, is defined as the set of all measurable functions $g$ such that 
$$
||g||_{(\Lambda^{p, \infty}_u(w))'}:= \sup_{f\in \Lambda^{p, \infty}_u(w)}\frac{\left|\displaystyle\int_{\mathbb R}f(x)g(x)u(x) dx\right| }{||f||_{\Lambda^{p, \infty}_u(w)}}<\infty.
$$
\end{definition}

In \cite{crs:crs},  these spaces were  characterized as follows:

\begin{proposition} \label{associate} \cite{crs:crs}
  If $0<p<\infty$, then 
$$
(\Lambda^{p, \infty}_u(w))'=\Lambda^{1}_u(W^{-1/p}).
$$
\end{proposition}

\begin{proposition}\label{NC}  \cite{acs:acs}
Assume that the Hilbert transform $H$ is well defined on $\Lambda^p_u(w)$ and that
\eqref{Hilbertantecedents} holds. Then,  we have the the following conditions: 

\smallskip

\noindent
{\rm (a)} $u\not\in L^1(\R)$ and $w\not\in L^1(\R^+)$.

\noindent
{\rm (b)} There exists $C>0$ such that, for every  measurable set $E$ and every interval $I$,  such that $E\subset I$,  we have that
$$
\frac{W(u( I ))}{W(u( E ))}\le C \left(\frac{|I|}{|E|}\right)^p.
$$
In particular, $W\circ u$ satisfies the doubling property; that is,  there exists a constant $c>0$ such that $W(u(2I))\leq c W(u(I))$,
for all intervals $I\subset\mathbb R$,  where $2I$ denotes the interval with the
same center as $I$ and double the size length. 

\noindent
{\rm (c)} $W$ is $p$ quasi-concave. In particular, $w\in\Delta_2$.

\noindent
{\rm (d)} $w\in B_{p, \infty}$.

\end{proposition}

As usual, we shall use the symbol $A\lesssim B$ to indicate that there exists a universal positive constant $C$,  independent of all important parameters,  such that $A\le C B$. $A\approx B$ means  that 
$A\lesssim B$ and $B\lesssim A$. 

Taking into account Proposition \ref{NC}, we shall assume  from now on, and without loss of generality, that 
$$
w\in \Delta_2, \quad u\notin L^1(\mathbb R) \quad\text{and}\quad w\notin L^1(\mathbb R^+).
$$

Also, we want to emphasize that, for a weight $u$ in $\mathbb R$ we say that $u$ satisfies the doubling property or $u\in\Delta_2$ if, for every interval $I$, $u(2I)\lesssim u(I)$, while in the case of a weight $w$ in $\mathbb R^+$, the condition $w\in\Delta_2$ is given by (\ref{d2}).

\,

Let us start by giving some important facts of each class of weights appearing in our  results.  

\section{Several classes of weights}

\subsection{The  $B^*_\infty$ class}

\ 

In this section we shall  study    weights satisfying (\ref{binfty}) and we shall prove  several properties that will be fundamental for our further results. 

\begin{lemma}\label{clickBstar}
Let $\varphi:(0,1]\to[0,1]$ be an increasing  submultiplicative function such that $\varphi(\lambda)<1$,  for some $\lambda\in(0,1)$. Then,
$$
\varphi(x)\lesssim \frac{1}{1+\log(1/x)}.
$$
\end{lemma}

\begin{proof}

Since $0<\lambda<1$, given $x\in (0, 1)$, there exists $k\in\mathbb N\cup\{0\}$ such that $x\in [\lambda^{k+1}, \lambda^k)$ and, using that $\varphi(\lambda)<1$, it is clear that 
$$
\sup_{j\in\mathbb N}  \varphi(\lambda)^j {\big(1+(j+1)\log(1/\lambda)\big)}=C_{\lambda}<\infty.
$$
Therefore, 
$$
\varphi(x)\le \varphi(\lambda^k)\le \varphi(\lambda)^k\lesssim \frac{1}{1+(k+1)\log(1/\lambda)}\lesssim  \frac{1}{1+\log(1/x)},
$$
as we wanted to see. 
\end{proof}

\begin{corollary}\label{elcor}
If $\varphi:(0,1]\to[0,1]$ is an increasing submultiplicative function, the following conditions are equivalent:
 
 \noindent
{\rm (1)} There exists $\lambda\in(0,1)$ such that $\varphi(\lambda)<1$.

	 \noindent
{\rm (2)}  $\varphi(x)\lesssim  \big(1+\log(1/x)\big)^{-1}$.

 \noindent
{\rm (3)} Given $p>0$, $\varphi(x)\lesssim  \big(1+\log(1/x)\big)^{-p}$.

 \noindent
{\rm (4)} $\displaystyle \lim_{x\to 0}\varphi(x)=0$.
\end{corollary}

\begin{proof} Clearly $(2)$, $(3)$ and $(4)$ imply $(1)$  and,    $(2)$  and $(3)$ imply $(4)$. On the other hand, by Lemma \ref{clickBstar}, $(1)$ implies $(2)$. Hence, it only remains to prove that $(1)$ implies $(3)$. Suppose that $\varphi(\lambda)<1$ and take $p>0$. If $\psi=\varphi^{1/p}$, then $\psi$ is also  increasing, submultiplicative and $\psi(\lambda)<1$, and by  Lemma \ref{clickBstar} we get  (3).
\end{proof}

In what follows, the following function will play an important role, 

\begin{equation*}\label{vb}
\overline W(t)=\sup_{s>0}\frac{W(st)}{W(s)}. 
\end{equation*}

\begin{proposition}\label{Bstar}
 The following statements are equivalent (see also \cite{k:k}): 
 
 \noindent
{\rm (i)} $w\in B^*_\infty$.

 \noindent
{\rm (ii)} There exists $\lambda\in(0,1)$ such that $\overline{W}(\lambda)<1$.
	
	 \noindent
{\rm (iii)}    $\displaystyle\frac{W(t)}{W(s)}\lesssim  \big(1+\log(s/t)\big)^{-1}$, for all $0<t\leq s$.
	
	 \noindent
{\rm (iv)} Given $p>0$, $\displaystyle\frac{W(t)}{W(s)}\lesssim  \big(1+\log(s/t)\big)^{-p}$,
	              for all $0<t\leq s$.
	              
	 \noindent
{\rm (v)} $\overline{W}(0^+)=0$.
  
   \noindent
{\rm (vi)}  For every $\varepsilon>0$, there exists $\delta>0$ such that $W(t)\leq \varepsilon W(s)$, provided $t\leq \delta s$.
\end{proposition}

\begin{proof}
Since  $\overline{W}$ is submultiplicative  we have,  by Corollary \ref{elcor} and letting $\varphi= \overline{W}_{|(0, 1]}$,  the  equivalences between (ii), (iii), (iv) and (v).  Also,  note that if (vi) holds,  then taking  $\lambda={t}/{s}$,  we get
${W(\lambda s)}\le\varepsilon {W(s)}$, for every $s\in[0,\infty)$ if $\lambda\le\delta$, and  hence we get (v). On the other hand, taking $t\leq \lambda s$, we get, by (v), that ${W(t)}\le\varepsilon {W(s)}$ whenever $t\leq \delta s$.

Now, if (i) holds,  for every $s\le r$, 
$$
W(s) \log \frac rs\le \int_s^r \frac{W(t)} t dt\lesssim W(r), 
$$
and since $W$ is increasing we deduce that  $W(s) (1+ \log \frac rs)\lesssim W(r)$, and (iii) holds.  On the other hand  if (iv) holds with $p=2$, then
$$
\int_0^r \frac{W(t)}t dt\lesssim W(r) \int_0^r   \Big(1+\log(r/t)\Big)^{-2} \frac{dt}t \lesssim W(r),
$$
and hence (i) holds. 
\end{proof}

\begin{proposition} \cite{k:k, n:n} \label{negBinfty}
Let $Q$ be the conjugate Hardy operator defined by 
$$
Qf(t)=\int_t^\infty f(s) \frac{ds}s.
$$
Then, for every  $0 < p < \infty$, 
 $$
 Q:L^p_{\text{\rm dec}}(w) \rightarrow   L^{p,\infty}(w) \quad\iff\quad w\in B^*_{\infty} \quad\iff\quad Q:L^p_{\text{\rm dec}}(w) \rightarrow   L^{p}(w). 
 $$
 \end{proposition}
 
 Using now interpolation on the cone of decreasing functions \cite{cm:cm}, we obtain the following corollary: 

\begin{corollary} \label{Qweakweak Bstar}
Let $0<p<\infty$. Then,
$$
w\in B^{*}_{\infty} \quad\iff\quad Q:L^{p,\infty}_{\text{\rm dec}}(w) \rightarrow L^{p,\infty}(w). 
$$
\end{corollary}

\subsection{The $B_{p, \infty}$ class}

\ 

As was mentioned in the introduction, if $p>1$, $w\in B_{p, \infty}$ if and only if   $w\in B_p$, and in this case the following result follows:

\begin{proposition} \label{lemaaa} If $1<p<\infty$ and $w\in B_{p, \infty}$,  then $$
||\chi_E||_{(\Lambda^{p, \infty}_u(w))'}\approx \frac{u(E)}{W^{1/p}(u(E))}. 
$$

\end{proposition}

\begin{proof}  
By Proposition \ref{associate},  we obtain that 
$$
||\chi_E||_{(\Lambda^{p, \infty}_u(w))'}=\int_0^{u(E)} \frac 1{W^{1/p}(t)} dt, 
$$
but,  since  $w\in B_p$, we have that  \cite{so:so}, 
$$
\int_0^{r} \frac 1{W^{1/p}(t)} dt\lesssim \frac{r}{W^{1/p}(r)}, 
$$
 and hence,  
$$
\frac{u(E)}{W^{1/p}(u(E))}\le \int_0^{u(E)} \frac 1{W^{1/p}(t)} dt\lesssim \frac{u(E)}{W^{1/p}(u(E))}, 
$$
as we wanted to see. 
\end{proof}

\subsection{$u\in A_\infty$  and $w\in B^*_\infty$}

\ 	
 
It is known that, if $u\in A_\infty$,  then there exists $q>1$ such that
\begin{equation} \label{Ainfty2}
 \frac{u(I)}{u(E)}\lesssim \left(\frac{|I|}{|E|}\right)^{q},
\end{equation}
for every interval $I$ and  every set $E\subset I$ \cite[p. 27]{jour}.

\begin{proposition} \label{enbuencamino} 
We have that $u\in A_{\infty}$ and  $w \in B^{*}_{\infty}$ if and only if the following condition holds: 
for every    $\varepsilon>0$,  there exists $0<\eta<1$ such that 
\begin{equation}\label{ab}
W\left(u(S)\right)\leq \varepsilon W\left(u(I)\right),
\end{equation}
for every interval $I$ and every measurable set $S\subseteq I$ satisfying that  $|S|\leq \eta |I|$.

\end{proposition}

\begin{proof}
 Let us first  assume that $w \in B^{*}_{\infty}$ and $u\in A_\infty$. Then, by Proposition~\ref{Bstar} we have that, for every $\varepsilon>0$,  there exists $\delta>0$ such that $W(t)\leq \varepsilon W(s)$,  whenever $t\leq \delta s$. 

On the other hand, if  $S\subset I$ is such that $|S|<\eta|I|$,  for some $\eta>0$, 
$$
\frac{u(S\big)}{u( I\big)}\le 
    C_u\ \left(\frac{|S|}{|I|}\right)^{r}
    <C_u\eta^r,
$$
where $r\in (0,1)$ and $C_u>0$ are constants depending on the  $A_{\infty}$ condition. So, choosing $\eta \in (0,1)$
such that $C_u\eta ^r< \delta$ we obtain the result. 

Conversely, let us see first that $u\in A_\infty$. Let $\varepsilon={1}/{2^{k-1}}$, with $k\in \N$ and let $\varepsilon'<{1}/{c^k}$, where $c>1$ is the constant in the $\Delta_2$ condition of $w$.  Let $\delta=\delta(\varepsilon')$ be such that, by hypothesis, $|S|\leq \delta |I|$ implies,
$$
W(u(  S))\leq \varepsilon' W(u(I)) < \frac{1}{c^k} W(u(I)).
$$
If $\displaystyle{\frac{u(I)}{u(S)}\leq 2^{k-1}}$ we get
$$
W(u(S))<\frac{1}{c^k} W\left(\frac{u(I)}{u(S)}u(S)\right)\leq \frac{1}{c}W(u(S)),
$$
 which is a contradiction. Hence, necessarily ${u(S)}\leq \frac{1}{2^{k-1}}u(I)=\varepsilon u(I)$. Thus,  we have proved that, 
 $$
 \forall \varepsilon>0, \ \exists \delta>0;  \ |S|\leq \delta |I| \implies  u(S)\leq \varepsilon u(I),
 $$
 and   this  implies that $u\in A_\infty$   \cite{gr:gr}. 
 
Let us now prove  that $w\in B_\infty^*$. By (\ref{ab}), we have that there exists $\lambda<1$ such that $W(u(E))/W(u(I)) <1/2$, provided $E\subset I$ and $|E|\le\lambda |I|$. 

Now, since $u\in A_\infty$ we have by (\ref{Ainfty2}),  that there exists $q>1$  and $C_u>0$ such that, for every $S\subset I$, 
\begin{equation}\label{ainf}
\frac{|S|}{|I|}\le C_u\bigg(\frac{u(S)}{u(I)}\bigg)^{1/q}, 
\end{equation}
and hence  if we take $\delta$ such that $C_u \delta^{1/q}\le \lambda$, and   $S\subset I$ such that ${u(S)}/{u(I)}\le \delta$, we obtain $W(u(S))/W(u(I)) <1/2$. 

Then,  if $0<t\le \delta s$ and we take an interval $I$ such that $u(I)=s$ and $S\subset I$ satisfies 
$u(S)=t$, we obtain $W(t)/W(s)<1/2$,  and consequently $\overline W(\delta)<1$. The result now follows from  
Proposition \ref{Bstar}. 
\end{proof}

\section{Main Results}

It is known  (see~\cite[pg. 256]{g:g}) that if $f\in \mathcal{C}^{\infty}_c$, then
\begin{equation} \label{cotlar}
(Hf)^2=f^2 + 2 H(fHf),
\end{equation}
and, using this equality, it was proved that, if $p>1$, 
$$
H:L^p\rightarrow L^p \implies H:L^{2p}\rightarrow L^{2p}. 
$$

Using the same sort of  ideas we obtain the following result: 

\begin{theorem} \label{ainfinito1}
 If \eqref{Hilbertantecedents} holds,  for some $0<p<\infty$  then, for every $r>p$, 
$$
H:\Lambda^r_u(w) \longrightarrow  \Lambda^{r}_u(w)
$$
is bounded. 
\end{theorem}

\begin{proof}
By (\ref{cotlar}), we have that 
\begin{align*}
||Hf||_{\Lambda^{2p,\infty}_u(w)} &= ||(Hf)^2||_{\Lambda^{p,\infty}_u(w)}^{1/2} = ||f^2 + 2 H(fHf)||_{\Lambda^{p,\infty}_u(w)}^{1/2}\\
       &\leq   C(||f^2||_{\Lambda^{p,\infty}_u(w)}+||H(fHf)||_{\Lambda^{p,\infty}_u(w)})^{1/2} \\
       &\leq   (C||f||_{{\Lambda^{2p,\infty}_u(w)}}^{2} +C_p ||fHf||_{\Lambda^{p}_u(w)})^{1/2}.
\end{align*}

Now, we have that 
$$
(fHf)^*_u (t) \leq f^*_u(t/2)  (Hf)^*_u(t/2),
$$
and hence, since $w\in \Delta_2$,  we obtain that
\begin{align*}
||fHf||_{\Lambda^{p}_u(w)}  &   \lesssim  \left(\int_0^{\infty} (f^*_u(t))^p ((Hf)^*_u(t))^p w(t) dt   \right)^{1/p} \\
  & = \left(\int_0^{\infty} \frac{(f^*_u(t))^p}{W^{1/2}(t)} ( W^{1/2p}(t) (Hf)^*_u(t))^p w(t)dt\right)^{1/p}\\
  & \leq ||Hf||_{\Lambda^{2p,\infty}_u(w)} ||f||_{\Lambda^{2p,p}_u(w)}, 
\end{align*}
where the $\Lambda^{q,p}_u(w)$ spaces are defined \cite{crs:crs} by the condition 
$$
||f||_{\Lambda^{q,p}_u(w)}=\bigg(\int_0^\infty f^*(t)^p W^{\frac pq-1}(t) w(t) dt \bigg)^{1/p}<\infty.
$$

Therefore, we have that
$$
||Hf||_{\Lambda^{2p,\infty}_u(w)} ^2\leq  C||f||_{{\Lambda^{2p,\infty}_u(w)}}^{2} +
C_p ||f||_{\Lambda_u^{2p,p}(w)} ||Hf||_{\Lambda_u^{2p,\infty}(w)}
$$
and, consequently,
$$
\frac{||Hf||^2_{\Lambda^{2p,\infty}_u(w)}}{||f||_{{\Lambda^{2p,p}_u(w)}}^{2}}
\leq
        C\frac{||f||_{{\Lambda^{2p,\infty}_u(w)}}^{2}}{||f||_{{\Lambda^{2p,p}_u(w)}}^{2}}
+C_p\frac{||Hf||_{\Lambda^{2p,\infty}_u(w)}}{||f||_{\Lambda^{2p,p}_u(w)}}.
$$
Using that $\Lambda^{2p,p}_u(w) \hookrightarrow  \Lambda^{2p,\infty}_u(w)$, we obtain that
$$
\bigg( \frac{||Hf||_{\Lambda^{2p,\infty}_u(w)}}{||f||_{{\Lambda^{2p,p}_u(w)}}} \bigg)^2
             \leq   C
           +C_p\frac{||Hf||_{\Lambda^{2p,\infty}_u(w)}}{||f||_{\Lambda^{2p,p}_u(w)}}, 
$$
from which it follows that
$$
||Hf||_{\Lambda^{2p,\infty}_u(w)} \lesssim  ||f||_{\Lambda^{2p,p}_u(w)}.
$$
and hence
$$
H: \Lambda^{2p,p}_u(w)\longrightarrow \Lambda^{2p,\infty}_u(w)
$$
is bounded. 
Finally,   by interpolation  (see \cite[Theorem 2.6.5]{crs:crs}), we  obtain that, for every $p<~r<2p$, 
$$
H:\Lambda^r_u(w)\to \Lambda^r_u(w)
$$
is bounded. The result now follows by iteration. 
\end{proof}

\begin{lemma} \label{lema1} Let $0<p<\infty$ be fixed. If {\rm (\ref{Hilbertantecedents})} holds,  then
$$
||H(uf)u^{-1}||_{(\Lambda^{p}_u(w))'}\lesssim ||f||_{(\Lambda^{p, \infty}_u(w))'}. 
$$
\end{lemma}

\begin{proof} The result follows easily  from  the definition of the associate spaces and the fact that 
$$\int_{\mathbb R} (Hf)(x) g(x) dx=-\int_{\mathbb R} (Hg)(x) f(x) dx.
$$
\end{proof}

\begin{lemma} \label{lema2}  If  $p>1$ and {\rm (\ref{Hilbertantecedents})} holds then, for every measurable set $E$, 
$$
\sup_F \frac{ \int_F  |H(u\chi_E)(x)| dx}{W^{1/p}(u(F))}\lesssim \frac{ u(E)}{W^{1/p}(u(E))},
$$
where the supremum is taken over all measurable sets $F$.
\end{lemma}

\begin{proof} Using duality and Lemma \ref{lema1}, we can prove  that (recall that $u(x)>0$, a.e.~$x\in~\mathbb R$):
\begin{eqnarray*}
 \int_F  |H(u\chi_E)(x)| dx&=&  \int_F  |H(u\chi_E)(x)u^{-1}(x)| u(x) dx
\\
&\le&  ||H(u\chi_E)u^{-1}||_{(\Lambda^{p}_u(w))'} ||\chi_F||_{\Lambda^{p}_u(w)}
\\
&\lesssim& 
||\chi_E||_{(\Lambda^{p, \infty}_u(w))'} ||\chi_F||_{\Lambda^{p}_u(w)},
\end{eqnarray*}
and the result follows by Proposition \ref{lemaaa}.

\end{proof}

As an immediate consequence, we obtain the following: 

\begin{corollary} If \eqref{Hilbertantecedents} holds for some $0<p<\infty$, then
\begin{equation}\label{cp}
\sup_{I} \frac 1{u(I)} \int_I |H(u\chi_I)(x)| dx <\infty, 
\end{equation}
where the supremum is taken over all intervals $I$.  
\end{corollary}

\begin{proof} By Theorem \ref{ainfinito1},  we can assume that $p>1$ and therefore Lemma \ref{lema2} holds. Taking $F=E=I$ in this lemma, we obtain the result. 
 \end{proof}

\begin{theorem} \label{ainfinito}
If $H$ satisfies \text{\rm (\ref{Hilbertantecedents})} for some $0<p<\infty$, then $u\in A_\infty$. 
\end{theorem}

\begin{proof} 
It is known that if
$$
Cf(\theta)=\text{p.v.} \int_0^1 \frac {f(x)}{\tan  {\pi} (\theta-x)} dx
$$
is   the conjugate operator,  then for  an $f\in L^1(0,1)$ such that  $Cf\in L^1(0,1)$, the non-tangential maximal operator $Nf\in L^1(0,1)$ \cite{bs:bs}. Moreover, if $f\ge 0$, it is also known   \cite{bs:bs} that $Nf\approx Mf$ and, in fact,  
$$
\int_0^1 Mf(x) dx\lesssim \int_0^1 f(x) dx + \int_0^1 |Cf(x)|dx \lesssim  \int_0^1 f(x) dx + \int_0^1 |Hf(x)|dx. 
$$

Now, if $f$ is supported in an interval $I=(a, b)$, we can consider $f_I$ defined on $(0,1)$ as
$f_I(x)= f( (b-a) x+a)$ and, by translation and dilation invariance of the operators $M$ and $H$, we have that  
$$
\frac 1{|I|}\int_I Mf(x) dx \lesssim 
\frac 1{|I|}\ \int_I f(x) dx + \frac 1{|I|}\ \int_I |Hf(x)| dx.
$$
Consequently, if we take $f=u\chi_I$ and use (\ref{cp}) we obtain that, for every interval $I$, 
$$
\int_I M(u\chi_I)(x) dx\lesssim u(I),
$$
and hence $u\in A_\infty$  \cite{w:w, hp:hp}.
\end{proof}

It was proved in \cite{acs:acs}  that if $u\in A_1$, the weak-type boundedness of $H$ implies that  $w\in B_\infty^*$. Now, an easy modification of that proof (we   include the details for the sake of completeness) also shows that if $u\in A_\infty$, the same results holds.

\begin{theorem}\label{pp} If $H$ satisfies \text{\rm (\ref{Hilbertantecedents})} for some $0<p<\infty$, then $w\in B_\infty^*$. 
\end{theorem}

\begin{proof} 
Let  $0<t \leq s<\infty$. Since $u\notin L^1(\mathbb R)$, there exists $\nu \in (0,1]$ and $b>0$  such that
$$
t=\int_{-b\nu}^{b\nu}u(r)\,dr \leq \int_{-b}^{b}u(r)\,dr =s.
$$
Now, simple computations of the Hilbert transform of the interval $(0, b)$ showed \cite{acs:acs} that, for every $b>0$, and every $\nu \in (0,1]$, 
\begin{equation}\label{log}
\frac{W\left(\int_{-b\nu}^{b\nu} u(s)\,ds\right)}{W\left(\int_{-b}^b u(s)\,ds\right)}\lesssim
               \left(1+ \log\frac{1}{\nu}\right)^{-p}.
\end{equation} and hence 
$$
\frac{W(t)}{W(s)}\lesssim \left(1+\log \frac{1}{\nu}\right)^{-p}.
$$
Let $S=(-b\nu, b\nu)$ and $I=(-b,b)$.
Since $u\in A_\infty$, we obtain by \eqref{ainf},  that there exists $q>1$ such that 
$$\nu=\frac{|S|}{|I|}\lesssim   \bigg( \frac{u(S)}{u(I)} \bigg)^{1/q}=  \bigg(\frac{t}{s} \bigg)^{1/q}
$$ and therefore 
$$
\frac{W(t)}{W(s)}\lesssim \left(1+\log \frac{s}{t}\right)^{-p}.
$$
From here, it follows by Proposition \ref{Bstar} that   $w\in B^*_{\infty}$.
 \end{proof}

\smallskip

Our next  goal   is to prove that 

\smallskip

$$
H:\Lambda_u^p(w) \to\Lambda_u^{p, \infty}(w)  \implies M:\Lambda_u^p(w) \to\Lambda_u^{p, \infty} (w). 
$$

\smallskip

Let us start with some previous lemmas. We need to introduce the following notation: given a finite family of disjoint intervals $\{I_i\}_i$, we shall denote by $I_i^*=101 I_i$. Then, 
$$
I_i^*=\bigcup_{j=-50}^{50} I_{i,j};  
$$
where $I_{i,j}$ is  the interval with   $|I_{i, j}|=|I_i|$,  
\begin{equation}\label{dist}
{\rm dist}( I_{i,j}, I_i)=(|j|-1)|I_i|, \qquad j\neq 0
\end{equation}\
and  such that $ I_{i,j}$ is situated to the left of $I_i$,  if $j<0$,  and to the right,  if $j>0$. Also, $I_{i, 0}=I_i$. 

If the family of intervals $\{I_i^*\}_i$ are pairwise disjoint, we say that  $\{I_i\}_i$ is {\it  well-separated}.

\begin{lemma}\label{cubos} Let $u  \in \Delta_2$. Then, given a well-separated finite family of intervals $\{I_i\}_i$,  it holds that 
$$
W^{1/p}\Big( u(\cup_i{I_{i, j_i}})\Big)\approx W^{1/p}\Big( u(\cup_i{I_i})\Big),
$$
for any choice of $j_i\in [-50, 50]$.

\end{lemma}

\begin{proof} Since $w$ is also in $\Delta_2$, we have that 
\begin{eqnarray*}
W^{1/p}\Big( u(\cup_i{I_{i, j_i}})\Big)&\le&  W^{1/p}\Big( u(\cup_i{I_i^*})\Big)=W^{1/p}\Big( \sum_i u({I_i^*})\Big)
\\
&\lesssim& W^{1/p}\Big( \sum_i u({I_i})\Big)= W^{1/p}\Big( u( \cup{I_i})\Big). 
\end{eqnarray*}

On the other hand,  $I_i \subset  I_{i, j_i}^*$ and hence
$$
 u(\cup_i I_i)=\sum_i u(I_i)\lesssim \sum_i u(I_{i, j_i}^*)\lesssim \sum_i u(I_{i, j_i})=  u(\cup_i I_{i, j_i})
 $$
 and therefore
 $$
 W^{1/p}\Big( u(\cup_i{I_i})\Big) \lesssim W^{1/p}\Big( u(\cup_i{I_{i, j_i}})\Big), 
 $$
 and the result follows.  
 \end{proof}

\begin{lemma} \label{ji} Let $f$ be a positive locally integrable function,  $\lambda>0$ and assume  $ \{I_i\}_{i=1}^m$ is a well separated family of intervals so that, for every $i$, 
$$
 \lambda\le \frac{ \int_{I_i} f(y) dy}{|I_i|}\le 2 \lambda. 
 $$

Then, for every $1\le i\le m$,   there exists $j_i\in [-50, 50]\setminus \{0\}$ such that
$$
\Big|H\big( f\chi_{\cup_{i=1}^m I_i}\big)(x)\Big|\ge \frac \lambda 8, \qquad \mbox {for every }\quad x\in \cup_{i\in J} I_{i, j_i}.
$$
\end{lemma}

\begin{proof}  Given $1\le i\le m$, let us define, for every $x\notin \cup_{i=1}^m I_{i}$, 
$$
A_i(x)=\sum_{j=1}^{i-1} \int_{I_j} \frac{f(y)}{x-y} dy, 
\qquad
B_i(x)=\sum_{j=i+1}^{m} \int_{I_j} \frac{f(y)}{x-y} dy, 
$$
and
$$
C_i(x)=A_i(x)+B_i(x).
$$
If we write $g= f\chi_{\cup_{i=1}^m I_i}$,  we have that 
$$
Hg(x)= C_i(x)+ \int_{I_i} \frac{f(y)}{x-y} dy.
$$
It also holds that if $I_i=(a_i, b_i)$, then $A_i$, $B_i$, and hence $C_i$, are decreasing functions in the interval $(b_{i-1}, a_i)$. 

Let us write $I_{i, -1}=(a_{i, -1}, b_{i, -1})$. 

\noindent
(a) If $ C_i (a_{i, -1}) \le   \lambda /4$, then $ C_i (x) \le   \lambda/ 4$, for every $x\in I_{i, -1}$ and since for these $x$,  
$$
\bigg|\int_{I_i } \frac{f(y)}{x-y} dy\bigg|= \int_{I_i } \frac{f(y)}{|x-y|} dy\ge \frac{\int_{I_i} f(y) dy}{2|I_i|}\ge\frac \lambda 2,
$$
we obtain that, for every $x\in I_{i, -1}$
$$
Hg(x)\le  \frac \lambda 4 - \frac \lambda 2=-\frac \lambda 4
$$
and consequently $|Hg(x)|\ge \frac \lambda 4$,  for every $x\in I_{i, -1}$. Hence, in this case, we choose $j_i=-1$. 

\noindent
(b) If $ C_i (a_{i, -1}) >   \lambda /4$, then  $ C_i (x) \ge   \lambda /4$, for every $x\in I_{i,  j}$ with $j\in[ -50, -2]$. Now, by (\ref{dist}),  we have that if if $x\in   I_{i,  j}$,
$$
\bigg|\int_{I_i } \frac{f(y)}{x-y} dy\bigg|= \int_{I_i } \frac{f(y)}{|x-y|} dy\le \frac{\int_{I_i} f(y) dy}{{\rm dist}(I_{i,j}, I_i)}\le \frac {2\lambda}{|j|-1},
$$
and thus, if we take $j=-17$, we obtain that, for every $x\in   I_{i,  -17}$
$$
Hg(x) \ge  \frac \lambda 4- \frac\lambda 8= \frac \lambda 8, 
$$
and consequently, in this case, with $j_i=-17$    the result follows. 
\end{proof}

\begin{theorem}\label{ref} If $p>0$, then
$$
H:\Lambda_u^p(w) \to\Lambda_u^{p, \infty}(w)  \implies M:\Lambda_u^p(w) \to\Lambda_u^{p, \infty} (w). 
$$

\end{theorem}

\begin{proof}  Let us  consider a positive locally integrable function $f$. Let $\lambda>0$ and let us take a compact set $K$
 such that $K\subset \{x: Mf(x)>\lambda\}$. Then, for each $x\in K$,  we can choose an interval $I_x$ such that
 \begin{equation*}\label{uuuu}
 \lambda<\frac{\int_{I_x} f(y) dy}{|I_x|}\le 2 \lambda. 
   \end{equation*}
 Then,  considering  $K\subset \cup_{x\in K} I_x^*$, we can obtain, using a Vitali covering lemma,  a well-separated finite family $\{I_{i}\}_{i=1}^m\subset \{I_{x}\}_x$, such that $K\subset \cup_i 3I_{i}^*$ and hence,
 \begin{equation}\label{K}
 W^{1/p} (u(K))\lesssim W^{1/p} (u(\cup_i 3I_{i}^*))\lesssim W^{1/p} (u(\cup_i I_{i})).
  \end{equation}

Now, by Lemma \ref{ji}, we obtain that there exists $j_i$ such that 
$$
\bigcup_{i=1}^m I_{i, j_i} \subset  \Big\{ \Big|H\big( f\chi_{\cup_{i=1}^m I_i}\big)(x)\Big|\ge \frac \lambda 8    \Big\} .
$$
Hence, by Lemma \ref{cubos}, we have that 
 \begin{eqnarray*}
 W\bigg(u\bigg(\bigcup_i I_{i}\bigg)\bigg)&\approx& W\bigg(u\bigg(\bigcup_{i=1}^m I_{i, j_i}\bigg)\bigg)\le  W\bigg(u\bigg(  \Big\{ \Big|H\big( f\chi_{\cup_{i=1}^m I_i}\big)(x)\Big|\ge \frac \lambda 8    \Big\}  \bigg)\bigg)
 \\
 & \lesssim& \frac 1{\lambda^p} ||f||_{\Lambda^p_u(w)}^p
 \end{eqnarray*}
and by (\ref{K}),  we obtain that 
$$
\lambda W^{1/p}(u(K)) \lesssim ||f||_{\Lambda_u^p(w) }. 
$$
Finally, the result follows by taking the supremum on all compact sets $K\subset \{ Mf>~\lambda\}$. 
\end{proof}

We finally present the   proof of our main Theorem \ref{main}. 

\begin{proof}[Proof of Theorem \ref{main}] If (\ref{Hilbertantecedents}) holds, then we have,  by Theorems \ref{ainfinito} and \ref{pp},  that $u\in A_\infty$ and  $w\in  B_\infty^*$. Also, by   Theorem  \ref{ref},   the weak-type boundedness of $M$ follows. 

Conversely,  it was proved  in \cite {bk:bk} that if  $u\in A_{\infty}$, 
\begin{equation*}\label{Hima}
(H^*f)^*_u(t)\lesssim \big(Q\,(Mf)^*_u\big) (t/4),
\end{equation*}
for all $t>0$, provided the right hand side is finite, where 
$$
H^*f(x)=\frac{1}{\pi} \sup_{\varepsilon> 0} \left|\int_{|x-y| > \varepsilon} \frac{f(y)}{x-y}\,dy\right|
$$
is  the Hilbert maximal operator. 
 Then, by Corollary \ref{Qweakweak Bstar} and the boundedness hypothesis on $M$, we have that 
\begin{align*}
||H^*f||_{\Lambda^{p, \infty}_u(w)}^{p}&\lesssim \sup_{t>0} W(t)^{1/p} Q(Mf)_u^*(t/4) \\
       &\lesssim  \sup_{t>0} W(t)^{1/p} (Mf)_u^*(t) 
\lesssim  \int_0^{\infty} f_u^* (t)^pw(t)\,dt,
\end{align*}
and therefore 
$$
H^*: \Lambda^p_u(w) \to \Lambda^{p, \infty}_u(w)
$$
is bounded. Now, since $C^\infty_c$ is dense in $\Lambda^p_u(w)$ and $Hf(x)$ is well defined at almost every point $x\in\mathbb R$, for every function  $f\in C^\infty_c$,  it follows by standard techniques that, for every $f\in \Lambda^p_u(w)$, $Hf(x)$ is well defined at almost every point $x\in\mathbb R$ and 
$$
H: \Lambda^p_u(w) \to \Lambda^{p, \infty}_u(w)
$$
is bounded, from which  the result follows. 
\end{proof}

Observe that we have also proved the following result:

\begin{theorem} \label{main3} If $0<p<\infty$, then
$$
H^*: \Lambda^p_u(w) \to \Lambda^{p, \infty}_u(w)
$$
is bounded if and only if  conditions {\rm (i)}, {\rm (ii)} and {\rm (iii)} of Theorem~\ref{main} hold. 
\end{theorem}

Taking into account Remark \ref{rarara} and Proposition \ref{enbuencamino}, we have the following characterization of \eqref{Hilbertantecedents}, in terms of geometric conditions on the weights,  in the case $0<p<1$. 

\begin{corollary} If $0<p<1$,  \eqref{Hilbertantecedents} holds if and only if $u\in A_\infty$, $w\in B^*_\infty$ and for every finite family of disjoint intervals $\{I_j\}_{j=1}^J$,
and every family of measurable sets $\{S_j\}_{j=1}^{J}$, with $S_j\subset I_j$, for every $j$,  we have that
\begin{equation}\label{raposo22}
\frac{W\left(u\left(\bigcup_{j=1}^J I_j\right)\right)}{W\left(u\left(\bigcup_{j=1}^J S_j\right)\right)}
       \leq C\max_{1\leq j\leq J} \left(\frac{|I_j|}{|S_j|}\right)^p, 
\end{equation}
or equivalently \eqref{raposo22} holds and, for every    $\varepsilon>0$,  there exists $0<\eta<1$ such that 
$$
W\left(u(S)\right)\leq \varepsilon W\left(u(I)\right),
$$
for every interval $I$ and every measurable set $S\subseteq I$ satisfying that  $|S|\leq \eta |I|$.
\end{corollary}

As mentioned in Remark  \ref{rarara}, the characterization of the weak-type boundedness of $M$ in the case $p\ge 1$ was left open in \cite{acs:acs} and it will be studied in a forthcoming paper. 

\ 

\centerline{\bf Application to the $L^{p.q}(u)$ spaces}

\

In the case of the Lorentz spaces $L^{p, q}(u)$ we observe that $L^{p, q}(u)=\Lambda^q_u(w)$ and $L^{p, \infty}(u)=\Lambda^{q, \infty}_u(w)$, with $w(t)=t^{ q/p -1}$ and since  in this case $w\in B^*_\infty$ and the boundedness of 
$$
M:L^{p,q}(u)\rightarrow L^{p,\infty}(u)
$$
is completely known (see \cite{crs:crs}, Theorem 3.6.1), we have  the following corollary,
extending the result of \cite[Theorem 5]{chk:chk} in the case of the Hilbert transform.

\begin{corollary}\label{312} For every $p, q>0$, 
$$
H: L^{p, q}(u)\longrightarrow L^{p, \infty}(u)
$$
is bounded if and only if $p\ge 1$ and

\noindent
{\rm (a)}  if  $p>1$ and $q>1$:  $u\in A_p$;

\noindent
{\rm (b)}  if $p>1$ and $q\le 1$:
$$
\frac{u(I)}{u(S)} \lesssim \bigg(\frac{|I|}{|S|}\bigg)^p
$$
for every measurable set $S\subset I$;

\noindent
{\rm (c)} if $p=1$, then necessarily $q\le 1$ and the condition is $u\in A_1$. 
\end{corollary}

\begin{remark}
We observe that Corollary~\ref{312}, together with Theorem~\ref{ref}, gives us that, if $p>1$, $q>1$ and $u\in A_p$, then $M:L^{p, q}(u)\longrightarrow L^{p, \infty}(u)$, which was proved in \cite{chk:chk}.
\end{remark}

\noindent
\textbf{Acknowledgments.} We would like to thank the referee for some useful comments which have improved the final version of this paper.

The first author would also like to thank the State Scholarship Foundation I.K.Y., of Greece.{\selectlanguage{greek} H olokl\'hrwsh ths erga$\sigma$\'ias aut\'hs \'egine sto pla\'isio 
ths ulopo\'ihshs tou metaptuqiako\'u progr\'ammatos pou sugqrhmatodot\'hjhke m\'esw ths 
Pr\'axhs <<Pr\'ogramma qor\'hghshs upotrofi\'wn I.K.U.  me diadika$\sigma$\'ia 
exatomikeum\'enhs axiol\'oghshs akad. \'etous 2011-2012>> ap\'o p\'orous tou 
E.P. <<Ekpa\'ideush kai dia b\'iou m\'ajhsh>> tou Eurwpaiko\'u koinwniko\'u tame\'iou
 (EKT) kai tou ESPA, tou 2007-2013.}


\begin{thebibliography}{99}


\bibitem{acs:acs}
E. Agora, M. J. Carro, and J. Soria, 
\textit{Boundedness of the Hilbert transform on weighted Lorentz spaces,}  J. Math. Anal. Appl. \textbf{395} (2012), 218--229.


\bibitem{k:k}
K. F. Andersen, 
\textit{Weighted generalized {H}ardy inequalities for nonincreasing
              functions, } Canad. J. Math.
              \textbf{43}  (1991),  no. 6,  {1121--1135}.


\bibitem{am:am}
M. A. Ari{\~n}o and B. Muckenhoupt, 
\textit{Maximal functions on classical Lorentz spaces and Hardy's
              inequality with weights for nonincreasing functions,} Trans. Amer. Math. Soc. \textbf{320} (1990), no. 2, 727--735.
              
              \bibitem{bk:bk}
R. Bagby and D. S.  Kurtz, 
\textit{A rearranged good $\lambda$ inequality,} Trans. Amer. Math. Soc. \textbf{293} (1986), no. 1, 71--81.

            

\bibitem{bs:bs}
C. Bennett and  R. Sharpley, 
\textit{Interpolation of Operators,} Pure and Applied Mathematics, \textbf{129}, Academic Press, Inc., Boston, MA, 1988.



\bibitem{crs:crs}
M. J. Carro, J. A. Raposo,  and J. Soria, 
\textit{Recent Developments in the Theory of Lorentz Spaces and Weighted Inequalities,} Mem. Amer. Math. Soc. \textbf{187} (2007), no. 877.

\bibitem{cm:cm}
 J. Cerd\`a    and J. Mart\'{\i}n, 
\textit{Interpolation restricted to decreasing functions and Lorentz spaces,} Proc. Edinburgh Math. Soc. (2) \textbf{42} (1999), no. 2, 243--256. 



\bibitem{chk:chk}
H. Chung, R. Hunt,  and D. S. Kurtz, 
\textit{The {H}ardy-{L}ittlewood maximal function on {$L(p,q)$} spaces with weights,} Indiana Univ. Math. J. \textbf{31} (1982), no. 1,  {109--120}. 



\bibitem{cf:cf}
R. Coifman  and  C. Fefferman, 
\textit{Weighted norm inequalities for maximal functions and singular integrals,} Studia Math. 
\textbf{51}  (1974), 
{241--250}. 


\bibitem{gr:gr}
J. Garc\'{\i}a Cuerva and  J. L. Rubio de Francia, 
\textit{Weighted Norm Inequalities and Related Topics}, North-Holland Mathematics Studies, 116. Notas de Matem\'atica [Mathematical Notes], 104. North-Holland Publishing Co., Amsterdam, 1985. 

\bibitem{g:g}
L. Grafakos, 
\textit{Classical Fourier Analysis}, Second edition. Graduate Texts in Mathematics, 249. Springer, New York, 2008.
 

\bibitem{hmw:hmw}
R. Hunt, B. Muckenhoupt, and R. Wheeden, 
\textit{Weighted norm inequalities for the conjugate function and
              Hilbert transform}, Trans. Amer. Math. Soc.
\textbf{176} (1973), 
{227--251}. 

\bibitem{hp:hp}
T. Hyt\"onen and C. P\'erez, 
\textit{Sharp weighted bounds involving $A_{\infty}$}, to appear in Anal. PDE.

\bibitem{jour}
J.-L. Journ\'e, 
\textit{Calder\'on-Zygmund operators, pseudodifferential operators and the Cauchy integral of Calder\'on}, Lecture Notes in Mathematics, \textbf{994},  Springer-Verlag, Berlin, 1983.

\bibitem{l1:l1}
G. Lorentz, 
\textit{Some new functional spaces}, Ann. of Math. (2)
\textbf{51} (1950), 
{37--55}.



\bibitem{l2:l2}
G. Lorentz, 
\textit{On the theory of spaces {$\Lambda$}}, Pacific J. Math.
\textbf{1} (1951), 
{411--429}. 


\bibitem{m:m}
B. Muckenhoupt, 
\textit{Weighted norm inequalities for the Hardy maximal function}, Trans. Amer. Math. Soc.
\textbf{165} (1972), 
{207--226}


\bibitem{n1:n1}
C. J. Neugebauer,
\textit{Weighted norm inequalities for averaging operators of monotone functions,} Publ. Mat. \textbf{35}  (1991), no. 2, 429--447.




\bibitem{n:n}
C. J. Neugebauer,
\textit{Some classical operators on Lorentz spaces,} Forum Math. \textbf{4}  (1992), no. 3, 135--146.


\bibitem{s:s}
E. Sawyer, 
\textit{Boundedness of classical operators on classical Lorentz spaces}, {Studia Math.}
\textbf{96} (1990), no. 2, 
{145--158}. 


\bibitem{so:so}
J. Soria, 
\textit{Lorentz spaces of weak-type}, {Quart. J. Math. Oxford Ser. (2)}
\textbf{49}  (1998), no. 193, 
{93--103}. 


\bibitem{w:w}
J. M. Wilson, 
\textit{Weighted inequalities for the dyadic square function without dyadic $A_\infty$}, {Duke Math. J. }
\textbf{55}  (1987), no. 1, 
{19--50}. 
 

\end{thebibliography}
\end{document}